\documentclass[a4paper,11pt]{amsart}
\usepackage[plainpages=false]{hyperref}
\usepackage{amsfonts, latexsym, rawfonts,amsmath,amssymb,amsthm, mathrsfs, lscape}
\usepackage{verbatim}

\usepackage[all]{xy}
\usepackage{graphicx,psfrag}

\usepackage{array, tabularx}

\usepackage{setspace}

\newtheorem{thm}{Theorem}

\newtheorem{cor}{Corollary}[section]

\newtheorem{prop}{Proposition}[section]
\newtheorem{prob}[thm]{Question}
\newtheorem{conj}[thm]{Conjecture}

\theoremstyle{remark}
\newtheorem{rmk}{Remark}[section]

\theoremstyle{definition}

\numberwithin{equation}{section}

\def\p{\partial}

\def\l{\lambda}
\def\L{\Lambda}

\def\i{\sqrt{-1}}

\def\cC{{\mathcal C}}

\def\cF{{\mathcal F}}

\def\cH{{\mathcal H}}

\def\cK{{\mathcal K}}

\begin{document}

\title[On the convergence of the Calabi flow]{On the convergence of the Calabi flow}

\author{Weiyong He}

\address{Department of Mathematics, University of Oregon, Eugene, Oregon, 97403}
\email{whe@uoregon.edu}

\begin{abstract}Let $(M, [\omega_0], J)$ be a compact Kahler manifold without holomorphic vector field. Suppose $\omega_0$ is (the unique) constant scalar curvature metric. We show that the Calabi flow with any smooth initial metric converges to the constant scalar curvature metric $\omega_0$ with the assumption that Ricci curvature stays uniformly bounded. 
\end{abstract}

\maketitle

\section{Introduction}

Let $(M, [\omega_0], J)$ be a compact Kahler  manifold. 
The space of Kahler potentials is given by
\[
\cH=\{\phi\in C^\infty: \omega=\omega_0+\i \p\bar \p \phi>0\}. 
\]
The Calabi flow is defined by E. Calabi  in his seminal paper \cite{Calabi82} as follows
\[
\frac{\p \phi}{\p t}=R_\phi-\underline{R},
\]
where $R_\phi$ denotes the scalar curvature of $\omega_\phi$ and $\underline{R}$ is the average of scalar curvature, a topological constant depending only on $(M, [\omega_0], J)$.

It is a natural equation to seek a canonical representative in a fixed Kahler class (called extremal metric in \cite{Calabi82}), which includes the Kahler metrics with constant scalar curvature as a special case. 
The Calabi flow is the gradient flow of Mabuchi's $K$-energy $K$ and is also a reduced gradient flow for the Calabi energy. One of the most challenging problem is whether the Calabi flow exists for all time or not (say with any smooth initial Kahler metric). X.X. Chen has made an ambitious conjecture as follows,

\begin{conj}[Chen]\label{C-1}The Calabi flow exists for all time for any initial potential in $\cH$.\end{conj}

Donaldson \cite{Donaldson03} gives a conjectural picture of the asymptotic behavior of the Calabi flow and relates it to the stability conjecture regarding the existence of constant scalar curvature metric. In particular, it is expected that when there exists a constant scalar curvature metric in $(M, [\omega_0], J)$, the Calabi flow exists for all time and converges to a metric with constant scalar curvature (it is contributed to Donaldson's conjecture in literature). Note that by Chen-Tian's theorem \cite{Chentian05, Chentian06}, constant scalar metrics (more generally extremal metrics) in $(M, [\omega_0], J)$ is unique up to diffeomorphism. The longtime existence problem is still largely open (complex dimension is 2 or higher) except for a few special cases.  In the present paper, we assume the following, \\

{\bf Assumption}: suppose the Calabi flow with any initial potential $\phi$ exists for all time with a uniform Ricci bound (for example, we assume a bound like $|Ric(t)|\leq C(\omega, \omega_\phi, |\phi|_{C^4})$ depending only on initial data and background geometry, but not on time). Note that by a result joint with X.X. Chen \cite{Chen-He}, the assumption on Ricci curvature actually implies the longtime existence.
\\

The main result in this paper is to prove the following,

\begin{thm}\label{T-1} With the assumption above, if there is a constant scalar curvature metric  in $(M, [\omega_0], J)$ and suppose there is no holomorphic vector field on $(M, J)$,   then the Calabi flow converges to a constant scalar curvature metric for any initial metric in $(M, [\omega_0], J)$.  
\end{thm}

\begin{rmk}The main motivation of the result is that Conjecture \ref{C-1} should imply that the Calabi flow converges to a constant scalar curvature metric when assuming such a metric exists. The Ricci curvature assumption is only technical in the proof. It will be interesting to drop this assumption. 
\end{rmk}

{\bf Acknowledgement:} I thank Bing Wang and Song Sun for valuable discussions, and I thank Prof. X.X. Chen for encouragements.   I am also grateful to J. Streets for sending me his recent preprint \cite{Streets} which inspires the consideration in Section 2. The author is partially supported by an NSF grant. 

\section{Evolution variational inequality along the Calabi flow}

In this section we prove a version of evolution variational inequality along the Calabi flow.

\begin{prop}Let $\phi_t$ be a smooth solution of the Calabi flow and $\psi\in \cH$ be any point (not on $\phi_t$). Then
\begin{equation}\label{EVI-1}
K(\psi)-K(\phi_t)\geq d(\psi, \phi_t) \frac{d}{dt} d(\psi, \phi_t)
\end{equation}
As a consequence, we have
\begin{equation}\label{EVI}
2s(K(\psi)-K(\phi_{t+s}))\geq d^2(\psi, \phi_{t+s})-d^{2}(\psi, \phi_t),
\end{equation}
where $d$ is a natural distance function on $\cH$ and it will be reviewed later. 
\end{prop}

Evolution variational inequality is well-known for the gradient flow of convex functionals on Hilbert spaces. In a recent preprint, J. Streets \cite{Streets} proved \eqref{EVI} in the frame work of $K$-energy minimizing movements (see Mayer \cite{Mayer02}, for example,  for some general results on minimizing movement of convex functionals on NPC spaces and Streets \cite{Streets12} for $K$-energy minimizing movement). As a consequence of this inequality, Streets also proved that the Calabi flow, when it exists for all time,  minimizes $\cK$-energy (when it is bounded below) and also minimizes the Calabi energy. 

We shall give an alternative proof of  the evolution inequality \eqref{EVI-1} and \eqref{EVI} is a direct consequence.
Our proof does not rely on the framework of minimizing movement. Actually it is rather straightforward given the results of Chen \cite{Chen00, Chen09} on  the geometric structure of $\cH$ and weak convexity of $K$-energy.
One motivation for this proof is to understand the similar picture for Kahler-Ricci flow on Fano manifolds.
In particular, we are interested in 

\begin{prob}\label{Q-3}Suppose $(M, [\omega_0], J)$ is a Fano manifolds and suppose $\cF$ functional is bounded below. Does the Kahler-Ricci flow minimizes $\cF$ functional? In particular, we would like to know whether  the Kahler-Ricci flow minimizes 
\[
H(\omega)=\int_M he^h \omega^n,
\]
where $h$ is the normalized Ricci potential such that $Ric(\omega)-\omega=\i \p\bar \p h$ and
\[
\int_M e^h \omega^n=\int_M \omega^n.
\]
\end{prob}

We first recall the metric structure on $\cH$, the space of Kahler potentials. Mabuchi \cite{Mabuchi86} introduced a Riemannian metric on $\cH$ and he also proves that this is formally a symmetric space with nonpositive curvature (See \cite{Semmes, Donaldson99} also). Chen \cite{Chen00} proved that $\cH$ is actually a metric space and it is convex by $C^{1, 1}$ geodesics (potential with bounded mixed derivative), by confirming partially conjectures of Donaldson \cite{Donaldson99}, in which he set up an ambitious program which tights up the existence of constant scalar curvature with the geometry of $\cH$.

Our proof of evolution variational inequality relies on two results of Chen \cite{Chen00, Chen09} and it is a rather straightforward consequence of his results.
The first result we need is the following weak convexity of $K$-energy proved in \cite{Chen09}.  
 Let $A(t), 0\leq t\leq 1$ be a geodesic in $\cH$ with two end points $A(0)=\phi, A(1)=\psi$, then $K$-energy is weakly convex in the sense that
\begin{equation}\label{E-chen1}
K(\psi)-K(\phi)\geq \frac{d}{dt} K(A(t))|_{t=0}=\int_M \dot A(0) (\underline{R}-R_{\phi})\omega_\phi^n. 
\end{equation}

The second result is about the derivative of distance function $d$ on $\cH$, which is proved in \cite{Chen00}. 
Suppose $\phi=\phi(s)$ is a smooth curve $C$ on $\cH$, then for any $\psi\in \cH$ not on the curve $C$, $d(\psi, \phi)$ is $C^1$ on $s$ and in particular,

\begin{equation}\label{E-chen2}
\frac{d}{ds}d(\psi, \phi)= -d(\psi, \phi)^{-1}\int_M \dot A(0) \frac{\p\phi}{\p s}  \omega_{\phi}^n,
\end{equation}
where $A(t)$ is the geodesic with $A(0)=\psi(s), A(1)=\psi$.

Given \eqref{E-chen1} and \eqref{E-chen2}, suppose $\phi(s)$ solves the Calabi flow. It follows that

\[
K(\psi)-K(\phi(s))-d(\psi, \phi(s))\frac{d}{ds}d(\psi, \phi(s))\geq \int_M \dot A(0) (\underline{R}-R_\phi+\frac{\p\phi}{\p s})\omega^n_{\phi(s)}=0
\]

This proved \eqref{EVI-1}, while \eqref{EVI} can be obtained directly by taking integration from $t$ to $t+s$, noting that $K$-energy is decreasing along the flow. 
As a corollary, it follows that

\begin{cor}For any initial potential $\phi_0\in \cH$, if the Calabi flow has a long time solution $\phi(t)$ such that $\phi(0)=\phi_0$, then 
\[\lim_{t\rightarrow \infty}K(\phi(t))=\inf_{\psi\in \cH}K(\psi).\]
\end{cor}

\begin{proof}This is proved in \cite{Streets}. Since the proof is rather short, we repeat it here. Suppose otherwise, then there exists $\delta>0$ and $\psi\in \cH$ such that for any $t$ sufficiently large,
\[
K(\psi)-K(\phi(t))\leq -\delta.
\]
In \eqref{EVI}, let $t=0$, $s\rightarrow \infty$, we get 
\[
-2s\delta \geq -d(\psi, \phi_0)^2.
\]Contradiction. 

\end{proof}

Suppose the Calabi flow $\phi(t)$ exists for all time with initial metric $\phi_0$. We denote
\[
A(\phi_0)=\lim_{t\rightarrow \infty}\cC(\phi(t))
\]  
Clearly the limit exists and $A(\phi_0)$ is a nonnegative number.

\begin{cor}\label{C-C}
Suppose $\phi(t)$ and $\psi(t)$ are two long-time solutions of the Calabi flow with initial metrics $\phi_0$ and $\psi_0$ respectively. Then
\[
A(\phi_0)=A(\psi_0). 
\]
In other words, if the Calabi flow exists for all time, then the Calabi energy has the same energy level at infinity. 
\end{cor}

\begin{proof}This is also proved in \cite{Streets}. 
We can assume that the curves $\phi(t)$ and $\psi(t)$ have no intersection. Otherwise we can assume that $\psi_0=\phi(t)$ for some $t$ for example. In this case it is clear that $A(\phi_0)=A(\psi_0)$.  
Suppose we have $A(\phi_0)=A(\psi_0)-3\delta$ for some $\delta>0$.  Since the statement only concerns the asymptotic behavior of the Calabi energy along $\phi(t)$ and  $\phi(t)$, we can assume that 
for any $t$ (otherwise let $t$ be sufficiently large),
\begin{equation}\label{E-2.5}
\cC(\phi_t)\leq \cC(\psi_t)-\delta
\end{equation}
Note that if $\phi(t)$ is a solution of Calabi flow, by \eqref{EVI}, we have, for $\psi$ not on the curve $\phi(t)$,
\begin{equation}
2s(K(\psi)-K(\phi(t+s)))\geq d^2(\psi, \phi(t+s))-d^2(\psi, \phi(t)).  
\end{equation}
Take $s=1$, and $\psi=\psi(t)$, we have
\begin{equation}\label{E-2.6}
d^2(\psi(t), \phi(t))+2(K(\psi(t))-K(\phi(t)))+2(K(\phi(t))-K(\phi(t+1)))\geq 0. 
\end{equation}
By a result of Calabi-Chen, we know that 
\[
d^2(\psi(t), \phi(t))\leq d^2(\phi_0, \psi_0).
\]
Also we have
\[
K(\phi(t))-K(\phi(t+1))=\int_t^{t+1} \cC(\phi(s)) ds\leq \cC(\phi_0). 
\]
On the other hand, 
\[
K(\psi_0)-K(\psi(t))=\int_0^t \cC(\phi(s))ds
\]
and 
\[
K(\phi_0)-K(\phi(t))=\int_0^t\cC(\psi(s))ds
\]
It follows that
\[
K(\psi(t))-K(\phi(t))=\int_0^t (\cC(\phi(s))-\cC(\psi(s)))ds+K(\psi_0)-K(\phi_0).
\]
By \eqref{E-2.5}, we have
\[
K(\psi(t))-K(\phi(t))\leq -\delta t +C 
\]
This contradicts \eqref{E-2.6} when $t$ is large enough. 
\end{proof}

\begin{rmk}
The results in this section are mostly proved by J. Streets \cite{Streets} using the framework of minimizing movement. Our main motivation in this section is to establish \eqref{EVI-1} along the Calabi flow using a direct approach. We hope this argument will give an approach to Question \ref{Q-3}. 
\end{rmk}

\section{The convergence of the Calabi flow}

We prove our main theorem in this section. Tian-Zhu \cite{TZ05} proved the convergence of Kahler-Ricci flow to a Kahler-Ricci soliton using a continuity method to deal with the Kahler-Ricci flow (see \cite{TZ3} also). Our method mimics their argument using continuity method and the level set of certain functional. In Kahler-Ricci flow they use Perelman's entropy and we use the Calabi energy here.

First we can establish a finite time stability result. Given a fixed background metric $\omega$, we define a set $B=B(\l, \L, K, \omega)$ as follows,
\[
B=\{\phi\in \cH: \l\omega\leq \omega_\phi\leq \L\omega, \|\phi\|_{C^{3, \alpha}}\leq K\}.
\]
In \cite{Chen-He}, we proved a short time existence result of the Calabi flow with smoothing property and smooth dependence of initial data (see Theorem 3.2 in \cite{Chen-He}). 
\begin{prop}\label{P-2}
Suppose the Calabi flow exists in the maximal time interval $[0, T)$ with initial potential $\phi_0$. Then for any $0<T_0<t$, there exists $\epsilon_0=\epsilon_0(\phi_0, T_0, \omega)$ such that
for any potential $|\psi-\phi_0|_{C^5}\leq \epsilon_0$, the Calabi flow exists with initial potential $\psi$ exists in $[0, T_0]$ and 
\[
|\psi(T_0)-\phi_0(T_0)|_{C^5}\leq C=C(\epsilon_0, \phi_0, T_0, \omega),
\] 
where $C\rightarrow 0$ when $\epsilon_0\rightarrow 0$. 
\end{prop}

\begin{proof}For any $T_0$ fixed, we know that $\phi_0(t)$ is a smooth path in $\cH$  for $t\in [0, T_0]$. We can then pick up uniform constants $\l, \L, K$ depending on $\phi_0, T_0$ such that
$\phi_0(t)\in B=B(\l, \L, K, \omega)$. 
We shall also assume that there exist a small number $\delta$ such that  if \[\min_{t\in [0, T_0]}|\psi-\phi_0(t)|_{C^{3, \alpha}} \leq \delta\]
then $\psi \in B$. 

By Theorem 3.2 in \cite{Chen-He}, we know  that for any potential $\psi\in B$, there exists  uniform constants $t_0, \epsilon_0$ and $C_0$ depending only on $\l, \L, K$ and $\omega$ such that the Calabi flow solution $\psi(t)$ exists for $[0, t_0]$; moreover,  if
\[
|\psi-\phi_0|_{C^{3, \alpha}} \leq \epsilon_0,
\]
then for any $t\in (0, t_0]$,
\begin{equation}\label{E-3.5}
\begin{split}
&|\psi(t)-\phi_0(t)|_{C^{3, \alpha}} \leq C_0 |\psi-\phi_0|_{C^{3, \alpha}}\\
&
|\psi(t)-\phi_0(t)|_{C^{4, \alpha}} \leq C_0 t^{-1/4} |\psi-\phi_0|_{C^{3, \alpha}}
\end{split}
\end{equation}
The smooth dependence part \eqref{E-3.5} is only emphasized around a constant scalar curvature metric in Theorem 3.2 \cite{Chen-He} (it is mainly used  in the paper to prove the stability of a constant scalar curvature metric along the Calabi flow),  but it holds around any smooth metric $\omega_{\phi_0}$ (it even holds  for initial metric in $C^{2, \alpha}$, see Theorem 3.1 in \cite{Chen-He} and Theorem 2.1 in \cite{He}).
We shall also assume $\epsilon_0\leq \delta$.  Let $k$ be the least integer such that $k\geq T_0/t_0$. Take $\epsilon$ small enough such that $(C_0)^k \epsilon<\epsilon_0$, then for any $\psi\in B$ satisfying
\[
|\psi-\phi_0|_{C^5}\leq \epsilon,
\]
we get that
\[
|\psi(t_0)-\phi_0(t_0)|_{C^{3, \alpha}}\leq C_0\epsilon.  
\]
Clearly $\psi(t_0)$ and $\phi_0(t_0)$ are both in $B$. We can then apply Theorem 3.2 in \cite{Chen-He} to the initial potential $\phi(t_0)$ and $\psi(t_0)$ to get a Calabi flow solution $\psi(t)$ in $[t_0, 2t_0]$. We can repeat this argument to get a solution $\psi(t)$ in $[0, kt_0]$ which contains $[0, T_0]$. (We emphasize that there is no contradiction since  $t_0$ depends on $T_0$  and we know that $T_0\leq kt_0<T$;  we cannot repeat the argument for $t>T_0$.) It then follows from \eqref{E-3.5} that, for any $t\in [0, T_0]$
\[
|\psi(t)-\phi_0(t)|_{C^{3, \alpha}}\leq (C_0)^k\epsilon<\epsilon_0. 
\]
By the smoothing property, we can get that
\[
|\psi(T_0)-\phi_0(T_0)|_{C^5}\leq C (\epsilon_0, \l, \L, K, T_0, \omega),
\]
where $C$ goes to zero when $\epsilon_0$ goes to zero. 
\end{proof}

\begin{rmk}With the assumption on Ricci curvature, then the above finite time stability becomes instant with the compactness theorem in \cite{Chen-He}. 
\end{rmk}

Now we are in the position to prove Theorem \ref{T-1}.

\begin{proof}Note that we assume there is no holomorphic vector field.  Let $\omega_0$ be the unique constant scalar curvature metric in $(M, [\omega_0])$ (uniqueness is proved by Chen-Tian).  
Let $\omega_\phi=\omega_0+\sqrt{-1}\p\bar \p \phi$ (we assume a normalization condition $I(\phi)=0$). Denote 
\[G=\{\phi\in \cH: |\phi(t)|_{C^5}\rightarrow 0, t\rightarrow \infty\}\]
to be the initial potentials such that $\phi(t)$ solves the Calabi flow equation and converges to zero (hence $\omega_{\phi(t)}$ converges to $\omega_0$). Clearly $0\in G$ and let $\phi\in \cH$. Let $\phi_s, 0\leq s\leq 1$ be a smooth path in $\cH$ such that $\phi_0=0, \phi_1=\phi$ (we can choose $\phi_s=s\phi$ say). Denote the corresponding  Calabi flow solution as $\phi_s(t)$. We want to show that $\phi_s(t)\in G$. Now let  
\[
S=\{s\in [0, 1]: \phi_s\in G\}. 
\]
We want to prove $S=[0, 1]$. By the method of continuity, we need to show $S$ is open and closed. The openness follows essentially from the finite time stability and the stability theorem around the constant scalar curvature metric proved in \cite{Chen-He}. In particular, we have

\begin{prop}Suppose $s\in S$, then for $\delta$ small enough $(s-\delta, s+\delta)\cap [0, 1] \in S$. 
\end{prop}

Since $\omega_{\phi_s(t)}$ converges to $\omega_0$ and hence $|\phi_s(t)|_{C^5}\rightarrow 0$ when $t$ sufficiently large. By stability of constant scalar curvature metric along the Calabi flow, we can assume that there exists $\epsilon=\epsilon(\omega_0)$ such that
\[
|\phi|_{C^5}\leq 2 \epsilon,
\]
then $\phi\in G$. Let $T$ be sufficiently large, we assume that $|\phi_s(T)|_{C^5}\leq \epsilon$. Fix $T$, by Proposition \ref{P-2}, there exists $\epsilon_1$ small enough such that $|\psi-\phi_s|_{C^5}\leq \epsilon_1$, then 
\[
|\phi(T)-\psi(T)|_{C^5}\leq c(\epsilon_1, \phi_s, T)
\]
We choose $\epsilon_1$ small enough such that $c(\epsilon_1, \phi_s, T)\leq \epsilon$.  Then
\[
|\psi(T)|\leq 2\epsilon. 
\]
Applying Theorem 4.1 in \cite{Chen-He} to  $\psi(T)$, we get that $\psi(T)\in G$. It then follows that, there exists $\epsilon_1$ such that $\psi\in G$ provided
\[
|\psi-\phi_s|_{C^5}\leq \epsilon_1.
\]
Hence if $s\in S$, we can choose $\delta$ small enough, such that $\phi_r\in G$ for $r\in (s-\delta, s+\delta)\cap [0, 1]$. \\

To show $S$ is closed, it suffices to show that if $[0, 1)\subset S$ then $1\in S$. Actually we want to show if $[0, 1)\subset S$, then the following uniform estimates for $\phi_s(t)$ for any $s\in [0, 1)$, 

\begin{prop}For any $\delta>0$ and $s\in [0, 1)$, there exists $T>0$ such that for any $t\geq T$, 
\[
|\phi_s(t)|_{C^5}\leq \delta. 
\]
\end{prop}

We argue by contradiction. Suppose otherwise, then there exists a sequence of $s_i\rightarrow 1$ and $t_i\rightarrow \infty$ such that
\[
|\phi_{s_i}(t_i)|_{C^5}\geq \delta
\]
for some $\delta>0$.
Since for any $s_i\in [0, 1)$, $|\phi_{s_i}(t)|_{C^5}\rightarrow 0$ when $t\rightarrow \infty$.  Hence we can choose $t_i$ such that 
\[
|\phi_{s_i}(t_i)|_{C^5}=\delta\geq |\phi_{s_i}(t)|_{C^5}, t\geq t_i.
\]
Let $\psi_i(t)=\phi_{s_i}(t_i+t)$ for $t\in [-1, 1]$. Then we know that $\psi_i(t)$ is a sequence of solutions of Calabi flow on $(M, [\omega_0], J)$ in $[0, 1]$ such that 
\[
|\psi_i(t)|_{C^5}\leq \delta. 
\]
By the compactness theorem, we get that $\psi_i(t)$ converges to $\psi_\infty(t)$ such that $\psi_\infty(t)$ is still a solution of the Calabi flow in $[0, 1]$. Note that at the moment, the convergence 
\[
\psi_i(t)\rightarrow \psi_\infty(t)
\]
is only in $C^{4, \alpha}$ for any $\alpha\in (0, 1)$. Now we claim that $\phi_\infty(t)\equiv 0$ for $t\in [0, 1]$. 
When $s=1$, let $\phi_1(t)=\phi(t)$ be the solution of the Calabi flow. By Corollary \ref{C-C}, we know that for any solution of the Calabi flow $\phi(t)$, the Calabi energy has the same energy level at infinity. Hence $\cC(\phi(t))\rightarrow 0$ when $t\rightarrow 0$. 
Hence for any $\epsilon>0$, we can find $t_0$ sufficiently large that 
\[
\cC(\phi(t_0))<\epsilon/2
\]
By Proposition \ref{P-2}, we get that for $i$ sufficiently large,
\[
\cC(\phi_{s_i}(t_0))<\epsilon.
\]
It then follows that for $i$ sufficiently large and $t\in [0, 1]$
\[
\cC(\psi_i(t))=\cC(\phi_{s_i}(t_i+t))<\epsilon.
\]
We can then conclude that for any $t\in [0, 1]$,
\[
\cC(\psi_\infty(t))=0
\]
Hence $\psi_\infty(t)=0$ are all potentials of the constant scalar curvature metric. By the uniqueness theorem, we know that $\psi_\infty(t)\equiv 0$ (we assume that there is no holomorphic vector field). In particular we know that
\[
\psi_i(t) \rightarrow 0
\]
in $C^{4, \alpha}$ for $i$ sufficiently large.  To get a contradiction, we use the assumption that the Ricci curvature is uniformly bounded for all $i$. 
Consider $\psi_i(t)$ in $[-1, 1]$. Then we know that (since the Ricci curvature, hence the scalar curvature is uniformly bounded), 
\[
|\p_t \psi_i(t)|\leq C
\] 
Hence it follows that $\psi_i(t)$ is uniformly bounded in $[-1, 1]$ ($\psi_i(0)$ is uniformly small in $C^5$). By the compactness theorem in \cite{Chen-He}, we know that $\psi_i(t)$ is uniformly bounded in $[-1, 1]$ in $C^{3, \alpha}$. 
Now applying smoothing property of the Calabi flow to $\psi_i(t)$ in $t\in [-1, 1]$ for  $i$ sufficiently large, it follows that for any $t\geq 0$,
\[
|\psi_i(t)|_{C^k}\leq C(k, t)
\]
In other words, $\psi_i(t)$ converges to $0$ in $C^\infty$ for any $t\geq 0$, when $i\rightarrow \infty$. This clearly contradicts the fact that $|\psi_i(0)|_{C^5}=\delta$. This completes the proof Theorem \ref{T-1}.
\end{proof}

\begin{rmk}The Ricci curvature assumption is only used to get an improved regularity at time $t=0$ for $\psi_i(t)$. Such an assumption is rather technical; we hope we can overcome the difficulty to drop the Ricci curvature assumption in future. The assumption on nonexistence of  holomorphic vector fields is rather superfluous and the similar strategy should also apply to extremal metrics. We will leave these problems for future study since Ricci curvature assumption seems to be more serious. 
\end{rmk}

\end{document}